\documentclass[11pt]{article}

\usepackage{amsfonts}
\usepackage{amssymb}
\usepackage{amsmath}
\usepackage{a4wide}
\usepackage{graphicx}
\usepackage{setspace}
\usepackage{amsthm}
\usepackage{hyperref}
\usepackage[labelfont=bf]{caption}
\theoremstyle{plain}
\newtheorem{thm}{Theorem}[section] 
\newtheorem{lem}[thm]{Lemma} 
\newtheorem{pro}[thm]{Proposition} 
\newtheorem{rem}[thm]{Remark} 
\newtheorem{co}[thm]{Corollary}

\newtheorem{conj}[thm]{Conjecture}
\makeatletter
\newcommand{\vast}{\bBigg@{3.2}}
\newcommand{\Vast}{\bBigg@{5}}
\makeatother
\begin{document}
\begin{center}
\section*{A density property for stochastic processes}
\subsection*{Riccardo Passeggeri\footnote{\noindent Email: riccardo.passeggeri@gmail.com}}
\subsection*{{\normalfont \large\textit{University of Toronto}}}
\today
\end{center}
\begin{abstract}
	Consider a class of probability distributions which is dense in the space of all probability distributions on $\mathbb{R}^{d}$ with respect to weak convergence, $\forall\,\,d\in\mathbb{N}$, and call it $\Phi$. Then, we construct various explicit classes of continuous (c\'{a}dl\'{a}g) processes, whose fdd belong to $\Phi$, that are dense in the space of all continuous (c\'{a}dl\'{a}g) processes with respect to convergence in distribution. This is motivated by the density of quasi-infinitely divisible (QID) distributions when $d=1$. If this result is extended to any $d\in\mathbb{N}$, then our result will imply that QID processes are dense in both spaces of continuous and c\'{a}dl\'{a}g processes.\\\\
	\textbf{Keywords:} stochastic processes, density property, dense class, quasi-infinitely divisible distributions.\\\\
	\textbf{MSC (2020):} 60G05, 60E05, 60G07, 60B10.
\end{abstract}
\section{Introduction}
This paper focuses on one of the fundamental properties in mathematics: density. Density has a pivotal role in mathematics and its application. For example, consider the Stone-Weierstrass theorem. Since polynomials are among the most tractable functions and computers can directly evaluate them, this theorem has both theoretical and practical relevance. An even more basic example is provided by the real numbers which, when endowed with the usual topology, have the rational numbers as a countable dense subset. We refer to classical books, like the ones of Bourbaki \cite{Bourbaki} and of Steen and Seebach \cite{SteenSseebach}, for further examples and for properties of dense sets.

In this work we present a density property for stochastic processes. To obtain it we proceed as follows. First, we consider any class of probability distributions on $\mathbb{R}^d$ which is dense in the space of all probability distribution on $\mathbb{R}^d$ with respect to weak convergence, for every $d\in\mathbb{N}$. 

Second, we consider the most relevant spaces of stochastic processes. In particular, we focus on stochastic processes with samples paths lying in one of the following three spaces of functions: $C([0,T])$, namely the space of continuous functions on the interval $[0,T]$, endowed with the uniform topology, $D([0,T])$, namely the space of c\'{a}dl\'{a}g functions on $[0,T]$, endowed with the Skorokhod ($J_1$) topology, and $D([0,\infty))$, namely the space of c\'{a}dl\'{a}g functions on $[0,\infty)$, endowed with the Skorokhod ($J_1$) topology. 

Third, we construct various explicit classes of stochastic processes whose finite dimensional distributions belong to the dense class of probability distributions. Each classes belongs to one of the three spaces of stochastic processes. 

Then, our main results state that each of these classes of processes is dense in the respective space of stochastic processes with respect to convergence in distribution.

Thus, our results demonstrate the following property of stochastic processes: ``density of probability distributions translates into density of stochastic processes". We actually show more. In addition to this property, we are also able to construct \textit{explicitly} these dense classes of stochastic processes.

From a distributional point of view, these results imply that any stochastic process can be approximated by an element of these explicit classes presented here. Thus, we do not lose any information when focusing on these classes instead of on the spaces of all stochastic processes.

These results are general and apply to any dense class of probability distributions, but they are also motivated by recent results on a particular class of probability distributions.

The class of infinitely divisible (ID) distributions is one of the most studied classes of probability distributions and their investigation dates back to the works of the father of modern probability: L\'{e}vy, Kolmogorov and De Finetti. Their characteristic function have a unique explicit formulation, called the L\'{e}vy-Khintchine formulation, in terms of three mathematical objects. These are the drift, which is a real valued constant, the Gaussian component, which is a non-negative constant, and the L\'{e}vy measure, which is a measure on $\mathbb{R}$ satisfying an integrability condition and with no mass at $\{0\}$. Gaussian and Poisson distributions are examples of this class.

In 2018, in \cite{LPS} the authors introduce the class of quasi-infinitely divisible (QID) distributions. A random variable $X$ is QID if and only if there exist two ID random variables $Y$ and $Z$ s.t.~$X+Y\stackrel{d}{=}Z$ and $Y$ is independent of $X$. QID distributions posses a L\'{e}vy-Khintchine formulation where the L\'{e}vy measure is now allowed to take negative values too.

One of the main results of \cite{LPS} is the density of QID distributions on $\mathbb{R}$ in the space of all probability distributions on $\mathbb{R}$ with respect to weak convergence. Since this work appeared, there has been a wide and increasing interests in QID distributions from both a theoretical and empirical perspective, see \cite{Passeggeri2} and references therein. In the infinite dimensional setting we have the works \cite{Passeggeri2,Passeggeri,Pass-bim} where the author introduces and investigates QID stochastic processes and QID random measures. In particular, in \cite{Passeggeri2} the author shows that QID completely random measures (CRMs) are dense in the space of all CRMs with respect to both weak and vague convergence.

Currently, various researchers are working on extending the density result of QID distributions from $\mathbb{R}$ to $\mathbb{R}^{d}$. If they succeed, by our results we would obtain that QID stochastic processes are dense in the space of all stochastic processes (with sample paths in $C([0,T])$, in $D([0,T])$, and in $D([0,\infty))$) with respect to convergence in distribution. We state this result after conjecturing that QID distributions are dense in $\mathbb{R}^{d}$.

We stress that our results are general and do not need such conjecture. Our results apply to \textit{any} family of distributions which is dense in the space of all probability distributions on $\mathbb{R}^{d}$, for every $d\in\mathbb{N}$. Moreover, we remark that our focus on the spaces of continuous and c\'{a}dl\'{a}g processes is motivated by the fact that they are the most theoretically studied and empirically used. However, the arguments that we adopt to prove the density results are quite general and, thus, we strongly believe that they can be used to prove density results for further spaces of stochastic processes and with further topologies (like the $S$ topology on the Skorokhod space introduced in \cite{Jakub}, see also \cite{Jakub2}).

The paper is structured as follows. In Section \ref{Sec-Prel}, we introduce preliminaries and notations. In Section \ref{Sec-Res}, we present the density result for stochastic processes with sample paths in $C([0,T])$, in $D([0,T])$ and in $D([0,\infty))$. In Section \ref{Sec-QID}, we apply them to the QID setting.
\section{Preliminaries}\label{Sec-Prel}
In this section we introduce some preliminaries and some of the notations used in the paper. First of all, we point out that in this work we use Billingsley's book \cite{Bill} as main reference, and adopt its notations.

Let $T>0$ and let $D[0,T]$ be the space of real functions on $[0,T]$ that are right-continuous and have left-hand limits. The space $D[0,T]$ is also called the Skorokhod space. Let $\Lambda_T$ denote the class of strictly increasing continuous mappings of $[0,T]$ onto itself. If $\lambda\in \Lambda_T$, then $\lambda0= 0$ and $\lambda T = T$. For $x,y\in D$, define
\begin{equation*}
d_{T}(x,y):=\inf\limits_{\lambda\in\Lambda_T}\Big\{\sup\limits_{t\in[0,T]}|\lambda t-t|\vee \sup\limits_{t\in[0,T]}|x(t)-y(\lambda t)|\Big\}.
\end{equation*}
This is the so-called Skorokhod metric or Skorokhod $J_1$ metric. The space $D[0,T]$ is not complete under $d_T$. We consider in $D[0,T]$ another metric $d_T^{\circ}$, which is equivalent to $d_T$, but under which $D[0,T]$ is complete. Let 
\begin{equation*}
\|\lambda\|_T^{\circ}:=\sup\limits_{0\leq s<t\leq T}\Big|\log\left(\frac{\lambda t-\lambda s}{t-s}\right)\Big|
\end{equation*}
and let 
\begin{equation*}
d^{\circ}_{T}(x,y):=\inf\limits_{\lambda\in\Lambda_T}\Big\{\|\lambda\|_T^{\circ}\vee \sup\limits_{t\in[0,T]}|x(t)-y(\lambda t)|\Big\}.
\end{equation*}
Sometimes we write $d(x,y)$ instead of $d_{T}(x,y)$ (and similarly for $d^{\circ}(x,y)$) when it does not create confusion. 

Further, we let $w'$ be the modified modulus of continuity in $D[0,1]$. Consider a set $\{t_0,...,t_v\}$, for some $v\in\mathbb{N}$, satisfying $0=t_0<t_1<...<t_v=1$ and call it $\delta$-sparse if it satisfies $\min\limits_{1\leq i\leq v}(t_i-t_{i-1}) > \delta$. Define, for $0 < \delta < 1$,
\begin{equation*}
w'(\delta)=w'(x,\delta)=\inf\limits_{\{t_0,...,t_v\}}\max\limits_{1\leq i\leq v}\sup\limits_{s,z\in[t_{i-1},t_i)}|x(s)-x(z)|,
\end{equation*}
where the infimum is taken over all the $\delta$-sparse sets.

Finally, we denote by $\stackrel{fdd}{\to},\stackrel{d}{\to}, \stackrel{p}{\to},\stackrel{a.s.}{\to},$ and $\stackrel{everywhere}{\to}$ the convergence in finite dimensional distributions, distribution, probability, almost surely, and everywhere, respectively, and let $\vee$ stand for maximum and $\wedge$ for minimum.
\section{Results}\label{Sec-Res}
In this section we investigate density results for stochastic processes with sample paths lying in three different spaces. The first space is $C([0,T])$, namely the space of continuous functions on $[0,T]$, endowed with the uniform topology. The second one is $D([0,T])$, namely the space of c\'{a}dl\'{a}g functions on $[0,T]$, endowed with the Skorokhod topology. The third one is $D([0,\infty))$, namely the space of c\'{a}dl\'{a}g functions on $[0,\infty)$, endowed with the Skorokhod topology.

Let $\Phi$ be a family of distributions which are dense in the space of all probability distributions on $\mathbb{R}^{d}$, for every $d\in\mathbb{N}$. We call a $\Phi$ stochastic process a process whose finite dimensional distributions belongs to $\Phi$.

Let use start with a preliminary result which will shorten the proofs of the main results: the density result for $\Phi$ times series with respect to finite dimensional distribution (fdd) convergence.
\begin{pro}\label{0}
The class of $\Phi$ time series is dense in the space of all time series with respect to the fdd convergence.
\end{pro}
\begin{proof} 	
	Let $X=(X_{t})_{t\in\mathbb{N}}$ be any stochastic process and let $\rho_{n}$ be the Prokhorov metric on measures on $\mathbb{R}^{n}$, for $n\in\mathbb{N}$. 
	Consider the sequence of $\Phi$ processes $Y^{(1)},Y^{(2)},...$ such that 
	\begin{equation*}
Y^{(n)}:=(Y^{(n)}_{1},Y^{(n)}_{2},...,Y^{(n)}_{n},0,...)
	\end{equation*}
	and that 
	\begin{equation*}
\rho_{n}\left(Y^{(n)},(X_{1},...,X_{n}) \right)< \frac{1}{n}.
	\end{equation*}
	This construction is possible thanks to the density of the family of distributions $\Phi$. Then, we have that 
	\begin{equation*}
	\lim\limits_{n\rightarrow\infty}\rho_{n}\left(Y^{(n)},(X_{1},...,X_{n}) \right)=0.
	\end{equation*} 
	
	We need to show that $\rho_{m}\left((Y_{1}^{(n)},...,Y_{m}^{(n)}),(X_{1},...,X_{m}) \right)< \frac{1}{n}$ for every $m\leq n$. However, this is true by definition of Prokhorov metric. In particular, let $\mu$ and $\nu$ be two finite dimensional distributions on $\mathbb{R}^{n}$. Then, by defining
	\begin{equation*}
	\rho_{n} (\mu ,\nu ):=\inf \left\{\varepsilon >0~|~\mu (A)\leq \nu (A^{{\varepsilon }})+\varepsilon \ {\text{and}}\ \nu (A)\leq \mu (A^{{\varepsilon }})+\varepsilon \ {\text{for all}}\ A\in {\mathcal {B}}(\mathbb{R}^{n})\right\}
	\end{equation*}
	we have that
		\begin{equation*}
			\rho_{n} (\mu ,\nu )\geq \inf \{\varepsilon >0~|~\mu (A)\leq \nu (A^{{\varepsilon }})+\varepsilon \ {\text{and}}\ \nu (A)\leq \mu (A^{{\varepsilon }})+\varepsilon \ {\text{for all}}\ A\in {\mathcal {B}}(\mathbb{R}^{n})~s.t.
		\end{equation*}
		\begin{equation*}
		~A=(B\times\mathbb{R}^{n-m}) \ {\text{where}} \ B\in {\mathcal {B}}(\mathbb{R}^{m}) \}
		\end{equation*}
	\begin{equation*}
	= \inf \{\varepsilon >0~|~\mu (A\times\mathbb{R}^{n-m})\leq \nu ((A\times\mathbb{R}^{n-m})^{\epsilon})+\varepsilon \ {\text{and}}\ 
	\end{equation*}
	\begin{equation*}
	\nu (A\times\mathbb{R}^{n-m})\leq \mu ((A\times\mathbb{R}^{n-m})^{\epsilon})+\varepsilon \ {\text{for all}}\ A\in {\mathcal {B}}(\mathbb{R}^{m})\}
	\end{equation*}
	\begin{equation*}
	=\inf \left\{\varepsilon >0~|~\mu_{|_{m}} (A)\leq \nu_{|_{m}} (A^{{\varepsilon }})+\varepsilon \ {\text{and}}\ \nu_{|_{m}} (A)\leq \mu_{|_{m}} (A^{{\varepsilon }})+\varepsilon \ {\text{for all}}\ A\in {\mathcal {B}}(\mathbb{R}^{m})\right\}
	\end{equation*}
	\begin{equation*}
	=\rho_{m} (\mu ,\nu ),
	\end{equation*}
where we used that $(A\times\mathbb{R}^{n-m})^{\epsilon}=A^{\epsilon}\times\mathbb{R}^{n-m}$ and where $\mu_{|_{m}}$ is the finite dimensional distributions on $\mathbb{R}^{m}$ s.t. $\mu(A\times \mathbb{R}^{n-m})=\mu(A)$ for every $A\in\mathcal{B}(\mathbb{R}^{m})$. Notice that such measure $\mu_{|_{m}}$ exists because of the consistency property of the finite dimensional distributions (see the Kolmogorov extension theorem).

Therefore, we have that for every $t_{1},...,t_{k}\in\mathbb{N}$ where $k\in\mathbb{N}$ we have that for every $n\geq\max\{t_{1},....,t_{k}\}$
\begin{equation*}
\rho_{k}\left((Y_{t_{1}}^{(n)},...,Y_{t_{k}}^{(n)}),(X_{t_{1}},...,X_{t_{k}}) \right)< \frac{1}{n},
\end{equation*}
and so for every $t_{1},...,t_{k}\in\mathbb{N}$ we have that
\begin{equation*}
\rho_{k}\left((Y_{t_{1}}^{(n)},...,Y_{t_{k}}^{(n)}),(X_{t_{1}},...,X_{t_{k}}) \right)\to0,
\end{equation*}
as $n\to\infty$. Therefore, we conclude that $Y^{(n)}\stackrel{fdd}{\to} X$.

In this proof we considered the class of stochastic process $(X_{t})_{t\in\mathbb{N}}$ for simplicity. However, the same arguments apply to any discrete parameters stochastic process, like $(X_{t})_{t\in a+b\mathbb{Z}}$, for $a\in\mathbb{R}$ and $b>0$. In particular, in this case we would have the sequence of $\Phi$ processes $a+b\mathbb{Z}$ on such that 
\begin{equation*}
Y^{(n)}:=(...,0,Y^{(n)}_{a+b(-n)},...,Y^{(n)}_{a+b(-1)},Y^{(n)}_{a},Y^{(n)}_{a+b(1)},...,Y^{(n)}_{a+b(n)},0,...)
\end{equation*}
and that
\begin{equation*}
\rho_{2n+1}\left(Y^{(n)},(X_{a+b(-n)},...,X_{a+b(n)}) \right)< \frac{1}{2n+1}.
\end{equation*}
Then, using the arguments above for every $t_{1},...,t_{k}\in\mathbb{Z}$ where $k\in\mathbb{N}$ we would have that for every $n\geq\max\{|t_{1}|,....,|t_{k}|\}$
\begin{equation*}
\rho_{k}\left((Y_{a+b(t_{1})}^{(n)},...,Y_{a+b(t_{k})}^{(n)}),(X_{a+b(t_{1})},...,X_{a+b(t_{k})}) \right)< \frac{1}{2n+1}.
\end{equation*}
\end{proof}

\begin{lem}\label{lem-1}
	Let $d\in\mathbb{N}$. Let $Z^{(n)}$, $n\in\mathbb{N}$, and $Z^{(n)}$ be random vectors on $\mathbb{R}^{d+1}$ such that $Z^{(n)}\stackrel{d}{\to}Z$ as $n\to\infty$. Denote the elements of $Z^{(n)}$ as follows $Z^{(n)}=(Z^{(n)}_{0},Z^{(n)}_{\frac{1}{d}},Z^{(n)}_{\frac{2}{d}},...,Z^{(n)}_{1})$. Let for $t\in[0,1]$
	\begin{equation*}
	\hat{Z}_{t}^{(n)}:=Z^{(n)}_{\frac{\lfloor td\rfloor}{d}}a_{t}+Z^{(n)}_{\frac{\lfloor td\rfloor+1}{d}}b_{t},\quad\text{and}\quad \hat{Z}_{t}:=Z_{\frac{\lfloor td\rfloor}{d}}a_{t}+Z_{\frac{\lfloor td\rfloor+1}{d}}b_{t},
	\end{equation*}
	where $a_{t}:=d(\frac{\lfloor td\rfloor+1}{d}-t)$ and $b_{t}:=d(t-\frac{\lfloor td\rfloor}{d})$. Then, for every fixed $\delta>0$, we have that
	\begin{equation*}
	\sup_{|t-s|<\delta}|\hat{Z}_{t}^{(n)}-\hat{Z}_{s}^{(n)}|\stackrel{d}{\to}	\sup_{|t-s|<\delta}|\hat{Z}_{t}-\hat{Z}_{s}|.
	\end{equation*}
\end{lem}
\begin{proof}
	Fix a $\delta>0$. By the continuous mapping theorem, it is enough to show that the function $g:\mathbb{R}^{d+1}\to[0,\infty)$ defined by $g(x)=\sup_{|t-s|<\delta}|\hat{x}_{t}-\hat{x}_{s}|$ is a continuous function. The function $g$ is well-defined and can be seen as $g=f\circ h$, where $h:\mathbb{R}^{d+1}\to C([0,1])$ ($h:x\mapsto \hat{x}$) and $f:C([0,1])\to[0,\infty)$ ($f:\hat{x}\mapsto\sup_{|t-s|<\delta}|\hat{x}_{t}-\hat{x}_{s}|$). Observe that for every $x,y\in\mathbb{R}^{d+1}$ we have that
	\begin{equation*}
	|g(x)-g(y)|=| \sup_{|t-s|<\delta}|\hat{x}_{t}-\hat{x}_{s}|-\sup_{|t-s|<\delta}|\hat{y}_{t}-\hat{y}_{s}||\leq \sup_{|t-s|<\delta}||\hat{x}_{t}-\hat{x}_{s}|-|\hat{y}_{t}-\hat{y}_{s}||
	\end{equation*}
	\begin{equation*}
	\leq\sup_{|t-s|<\delta}|\hat{x}_{t}-\hat{y}_{t}-\hat{x}_{s}+\hat{y}_{s}|\leq 2\sup_{t}|\hat{x}_{t}-\hat{y}_{t}|= 2\max_{t=0,\frac{1}{d},\frac{2}{d},...,1}|x_{t}-y_{t}|=\|x-y\|_{\infty},
	\end{equation*}	
	where we used the fact the supremum distance for the difference of linear interpolations is obtained on the grid points. Thus, $g$ is uniformly continuous, hence continuous.
\end{proof}
In the following we obtain the density result for the class of processes with paths in $C([0,T])$ endowed with the uniform topology.
\begin{thm}\label{T-C}
Let $T>0$. The class of processes build as linear interpolation of $\Phi$ random vectors is dense in the space of stochastic processes with paths in $C([0,T])$ endowed with the uniform topology with respect to weak convergence.
\end{thm}
\begin{proof}
	Let $(X_{t})_{t\in[0,1]}$ be any stochastic process with continuous paths. We focus on the interval $[0,1]$ but the same arguments of the proof apply to any interval $[0,T]$, for $T>0$.
	
	Let $\delta_{m}:=\frac{1}{2^{m}}$, $m\in\mathbb{N}$. For every $n\in\mathbb{N}$, consider a $(2^{n}+1)$-dimensional $\Phi$ random vector 
	\begin{equation*}
	Y^{(n)}:=(Y^{(n)}_{0},Y^{(n)}_{\frac{1}{2^{n}}},Y^{(n)}_{\frac{2}{2^{n}}},...,Y^{(n)}_{1})
	\end{equation*}
	such that
	\begin{equation*}
	\rho_{2^{n}+1}\left(Y^{(n)},(X_{0},X_{\frac{1}{2^{n}}},...,X_{1}) \right)< \frac{1}{n},
	\end{equation*}
	and 
	\begin{equation}\label{max-property-Yn}
	\max_{m<n}\rho(\sup_{|t-s|<\delta_{m}}|\hat{Y}_{t}^{(n)}-\hat{Y}_{s}^{(n)}|,\sup_{|t-s|<\delta_{m}}|\hat{X}_{t}^{(n)}-\hat{X}_{s}^{(n)}|)<\frac{1}{n},
	\end{equation}
	where
	\begin{equation*}
	\hat{Y}_{t}^{(n)}:=Y^{(n)}_{\frac{\lfloor t2^{n}\rfloor}{2^{n}}}a^{(n)}_{t}+Y^{(n)}_{\frac{\lfloor t2^{n}\rfloor+1}{2^{n}}}b^{(n)}_{t},\quad\textnormal{and}\quad \hat{X}_{t}^{(n)}:=X_{\frac{\lfloor t2^{n}\rfloor}{2^{n}}}a^{(n)}_{t}+X_{\frac{\lfloor t2^{n}\rfloor+1}{2^{n}}}b^{(n)}_{t},
	\end{equation*}
	where $a^{(n)}_{t}:=2^{n}(\frac{\lfloor t2^{n}\rfloor+1}{2^{n}}-t)$ and $b^{(n)}_{t}:=2^{n}(t-\frac{\lfloor t2^{n}\rfloor}{2^{n}})$.

	The existence of $Y^{(n)}$ is ensured by the density of the family of distributions $\Phi$ and by Lemma \ref{lem-1}. Further, observe that $(\hat{Y}_{t}^{(n)})_{t\in[0,1]}$ has continuous paths. 
	
	For any $k\in\mathbb{N}$ and $t_{1},...,t_{k}\in[0,1]$, we now prove that 
	\begin{equation*}
	\rho_{k}\left((\hat{Y}^{(n)}_{t_{1}},...,\hat{Y}^{(n)}_{t_{k}}),(X_{t_{1}},...,X_{t_{k}}) \right)\rightarrow 0,\quad\textnormal{as $n\rightarrow\infty$}.
	\end{equation*}
	Consider the following similar linear piecewise interpolation for $X_{t}$:
\begin{equation*}
\hat{X}_{t}^{(n)}:=X_{\frac{\lfloor t2^{n}\rfloor}{2^{n}}}a^{(n)}_{t}+X_{\frac{\lfloor t2^{n}\rfloor+1}{2^{n}}}b^{(n)}_{t}.
\end{equation*}
By triangular inequality we have that
\begin{equation*}
\rho_{k}\left((\hat{Y}^{(n)}_{t_{1}},...,\hat{Y}^{(n)}_{t_{k}}),(X_{t_{1}},...,X_{t_{k}}) \right)
\end{equation*}
\begin{equation*}
\leq \rho_{k}\left((\hat{Y}^{(n)}_{t_{1}},...,\hat{Y}^{(n)}_{t_{k}}),(\hat{X}_{t_{1}}^{(n)},...,\hat{X}_{t_{k}}^{(n)}) \right)+\rho_{k}\left((\hat{X}^{(n)}_{t_{1}},...,\hat{X}^{(n)}_{t_{k}}),(X_{t_{1}},...,X_{t_{k}}) \right).
\end{equation*}
As mentioned above, we know that 
\begin{equation*}
\rho_{2^{n}+1}\left(Y^{(n)},(X_{0},X_{\frac{1}{2^{n}}},...,X_{1}) \right)< \frac{1}{n},
\end{equation*}
for every $n\in\mathbb{N}$. 

Thus, for every $t_{1},...,t_{k}\in[0,1]$, using the arguments of the proof of Theorem \ref{0}, mainly the ones on the properties of the Prokhorov metric, we have that 
\begin{equation*}
\rho_{2k}\left((Y^{(n)}_{\frac{\lfloor t_{1}2^{n}\rfloor}{2^{n}}},...,Y^{(n)}_{\frac{\lfloor t_{k}2^{n}\rfloor}{2^{n}}},Y^{(n)}_{\frac{\lfloor t_{1}2^{n}\rfloor+1}{2^{n}}},...,Y^{(n)}_{\frac{\lfloor t_{k}2^{n}\rfloor+1}{2^{n}}}),(X_{\frac{\lfloor t_{1}2^{n}\rfloor}{2^{n}}},...,X_{\frac{\lfloor t_{k}2^{n}\rfloor}{2^{n}}},X_{\frac{\lfloor t_{1}2^{n}\rfloor+1}{2^{n}}},...,X_{\frac{\lfloor t_{k}2^{n}\rfloor+1}{2^{n}}}) \right)
\end{equation*}
\begin{equation*}
< \frac{1}{n},
\end{equation*}
for $n$ large enough, namely for $2^{n}+1\geq 2k$, and so it converges to zero as $n\to\infty$.

Furthermore, it is possible to see that, by the continuity of the paths of the process $(X_{t})_{t\in[0,1]}$, we obtain that
\begin{equation*}
\rho_{2k}\left((X_{\frac{\lfloor t_{1}2^{n}\rfloor}{2^{n}}},...,X_{\frac{\lfloor t_{k}2^{n}\rfloor}{2^{n}}},X_{\frac{\lfloor t_{1}2^{n}\rfloor+1}{2^{n}}},...,X_{\frac{\lfloor t_{k}2^{n}\rfloor+1}{2^{n}}}),(X_{t_{1}},...,X_{t_{k}},X_{t_{1}},...,X_{t_{k}})\right)\rightarrow0,
\end{equation*}
as $n\rightarrow\infty$. Therefore, by triangular inequality we have that
\begin{equation}\label{conv}
\rho_{2k}\left((Y^{(n)}_{\frac{\lfloor t_{1}2^{n}\rfloor}{2^{n}}},...,Y^{(n)}_{\frac{\lfloor t_{k}2^{n}\rfloor}{2^{n}}},Y^{(n)}_{\frac{\lfloor t_{1}2^{n}\rfloor+1}{2^{n}}},...,Y^{(n)}_{\frac{\lfloor t_{k}2^{n}\rfloor+1}{2^{n}}}),(X_{t_{1}},...,X_{t_{k}},X_{t_{1}},...,X_{t_{k}})\right)\rightarrow0,
\end{equation}
as $n\rightarrow\infty$. Then, by the continuous mapping theorem we obtain that \begin{equation*}
\rho_{k}\left((\hat{Y}^{(n)}_{t_{1}},...,\hat{Y}^{(n)}_{t_{k}}),(X_{t_{1}},...,X_{t_{k}}) \right)\to0,\quad \textnormal{as $n\rightarrow\infty$}.
\end{equation*}

Since $t_{1},...,t_{k}$ were any times in $[0,1]$, we conclude that $(\hat{Y}_{t}^{(n)})_{t\in[0,1]}\stackrel{fdd}{\rightarrow}(X_{t})_{t\in[0,1]}$, as $n\rightarrow\infty$.

Let us now prove tightness. First, observe that for every $\delta>0$ we have that
\begin{equation*}
|\sup_{|t-s|<\delta}|\hat{X}_{t}^{(n)}-\hat{X}_{s}^{(n)}|-\sup_{|t-s|<\delta}|X_{t}-X_{s}||\leq 2\sup_{t}|\hat{X}_{t}^{(n)}-X_{t}|
\end{equation*}
\begin{equation}\label{Xn}
\leq 4 \sup_{|t-s|<2^{-n}}|X_{t}-X_{s}|\stackrel{everywhere}{\to} 0\quad\textnormal{as $n\to\infty$,}
\end{equation}
because $X$ is an element of $C$ (see also the proof of Theorem 7.5 in \cite{Bill}). Thus, also observe that
\begin{equation}\label{delta}
\sup_{|t-s|<\delta}|X_{t}-X_{s}|\stackrel{everywhere}{\to} 0,\quad\textnormal{as $\delta\to 0$.}
\end{equation}
Furthermore, notice that by the definition of $Y^{(n)}$ (in particular property (\ref{max-property-Yn})) and the definition of the L\'{e}vy-Prokhorov metric we deduce that for every $m,n\in\mathbb{N}$ with $m<n$ and every $x\in[0,\infty)$
\begin{equation*}
\mathbb{P}(\sup_{|t-s|<\delta_{m}}|\hat{Y}_{t}^{(n)}-\hat{Y}_{s}^{(n)}|<x)\geq \mathbb{P}(\sup_{|t-s|<\delta_{m}}|\hat{X}_{t}^{(n)}-\hat{X}_{s}^{(n)}|<x-\frac{1}{n})-\frac{1}{n}
\end{equation*}
\begin{equation}\label{Yn}
\geq\mathbb{P}(\sup_{|t-s|<\delta_{m}}|\hat{X}_{t}^{(n)}-\hat{X}_{s}^{(n)}|<x-\frac{1}{\tilde{n}})-\frac{1}{\tilde{n}},
\end{equation} 
for every $\tilde{n}\leq n$.

In order to prove tightness we prove that: for every $\varepsilon$ and $\eta$, there exist $m,n_{0}\in\mathbb{N}$ such that
\begin{equation*}
\mathbb{P}(\sup_{|t-s|<\delta_{m}}|\hat{Y}_{t}^{(n)}-\hat{Y}_{s}^{(n)}|\geq\varepsilon)\leq \eta,\quad n\geq n_{0},
\end{equation*}
which is equivalent to
\begin{equation*}
\mathbb{P}(\sup_{|t-s|<\delta_{m}}|\hat{Y}_{t}^{(n)}-\hat{Y}_{s}^{(n)}|<\varepsilon)\geq 1-\eta, \quad n\geq n_{0}.
\end{equation*}

So fix $\varepsilon$ and $\eta$. Choose $m$ such that $\mathbb{P}(\sup_{|t-s|<\delta_{m}}|X_{t}-X_{s}|<\frac{\varepsilon}{2})\geq 1-\frac{\eta}{2}$. This is possible thanks to (\ref{delta}). Moreover, by (\ref{Xn}) we have that there exists an $n^{*}\in\mathbb{N}$ such that $\frac{1}{n^{*}}<\frac{\eta}{2}$, $\frac{1}{n^{*}}<\frac{\varepsilon}{2}$, and
\begin{equation*}
\mathbb{P}\bigg(\sup_{|t-s|<\delta_{m}}|\hat{X}_{t}^{(n)}-\hat{X}_{s}^{(n)}|<\frac{\varepsilon}{2}+\Big(\frac{\varepsilon}{2}-\frac{1}{n^{*}}\Big)\bigg)+\Big(\frac{\eta}{2}-\frac{1}{n^{*}}\Big)\geq 1-\frac{\eta}{2}
\end{equation*}
\begin{equation*}
\Leftrightarrow\mathbb{P}(\sup_{|t-s|<\delta_{m}}|\hat{X}_{t}^{(n)}-\hat{X}_{s}^{(n)}|<\varepsilon-\frac{1}{n^{*}})-\frac{1}{n^{*}}\geq 1-\eta,
\end{equation*}
for every $n\geq n^{*}$. Finally, let $n_{0}=\max(m+1,n^{*})$, then by (\ref{Yn}) we have that
\begin{equation*}
\mathbb{P}(\sup_{|t-s|<\delta_{m}}|\hat{Y}_{t}^{(n)}-\hat{Y}_{s}^{(n)}|<\varepsilon)\geq 1-\eta,
\end{equation*}
for every $n\geq n_{0}$.

We conclude the proof by applying Theorem 7.5 in \cite{Bill}.
\end{proof}
\begin{rem}
	The proof holds also in the case the partition has size $\frac{1}{n}$ instead of $\frac{1}{2^{n}}$.
\end{rem}

\begin{lem}\label{lem-2}
	Let $d\in\mathbb{N}$. Let $Z^{(n)}$, $n\in\mathbb{N}$, and $Z^{(n)}$ be random vectors on $\mathbb{R}^{d+1}$ such that $Z^{(n)}\stackrel{d}{\to}Z$ as $n\to\infty$. Denote the elements of $Z^{(n)}$ as follows $Z^{(n)}=(Z^{(n)}_{0},Z^{(n)}_{\frac{1}{d}},Z^{(n)}_{\frac{2}{d}},...,Z^{(n)}_{1})$. For $t\in[0,1]$, let $\hat{Z}_{t}^{(n)}:=Z^{(n)}_{\frac{\lfloor td\rfloor}{d}}$ and $\hat{Z}_{t}:=Z_{\frac{\lfloor td\rfloor}{d}}$. Then, for every fixed $\delta>0$, we have that as $n\to\infty$
	\begin{equation*}
	\sup_{\substack{t_{1}\leq t\leq t_{2}\\t_{2}-t_{1}\leq\delta}}|\hat{Z}_{t_{2}}^{(n)}-\hat{Z}_{t}^{(n)}|\wedge|\hat{Z}_{t}^{(n)}-\hat{Z}_{t_{1}}^{(n)}|\stackrel{d}{\to}	\sup_{\substack{t_{1}\leq t\leq t_{2}\\t_{2}-t_{1}\leq\delta}}|\hat{Z}_{t_{2}}-\hat{Z}_{t}|\wedge|\hat{Z}_{t}-\hat{Z}_{t_{1}}|,
	\end{equation*}
	\begin{equation*}
	|\hat{Z}_{\delta}^{(n)}-\hat{Z}_{0}^{(n)}|\stackrel{d}{\to}	|\hat{Z}_{\delta}-\hat{Z}_{0}|,\quad |\hat{Z}_{\frac{d-1}{d}}^{(n)}-\hat{Z}_{1-\delta}^{(n)}|\stackrel{d}{\to}|\hat{Z}_{\frac{d-1}{d}}-\hat{Z}_{1-\delta}|,\quad\text{and}\quad
	\sup_{t}|\hat{Z}_{t}^{(n)}|\stackrel{d}{\to}	\sup_{t}|\hat{Z}_{t}|.
	\end{equation*}
\end{lem}
\begin{proof}
	Fix a $\delta>0$. By the continuous mapping theorem, it is enough to show that the function $g_{1},g_{2},g_{3},g_{4}:\mathbb{R}^{d+1}\to[0,\infty)$ defined by $g_{1}(x)=\sup_{\substack{t_{1}\leq t\leq t_{2}\\t_{2}-t_{1}\leq\delta}}|\hat{x}_{t_{2}}-\hat{x}_{t}|\wedge|\hat{x}_{t}-\hat{x}_{t_{1}}|$, $g_{2}=|\hat{x}_{\delta}-\hat{x}_{0}|$, $g_{3}=|\hat{x}_{\frac{d-1}{d}}-\hat{x}_{1-\delta}|$, and $g_{4}=\sup_{t}|\hat{x}_{t}|$ are continuous. Let us start with $g_{1}$. Observe that for every $x,y\in\mathbb{R}^{d+1}$ we have the following
	\begin{equation*}
	|g_{1}(x)-g_{1}(y)|=|\sup_{\substack{t_{1}\leq t\leq t_{2}\\t_{2}-t_{1}\leq\delta}}|\hat{x}_{t_{2}}-\hat{x}_{t}|\wedge|\hat{x}_{t}-\hat{x}_{t_{1}}|-\sup_{\substack{t_{1}\leq t\leq t_{2}\\t_{2}-t_{1}\leq\delta}}|\hat{y}_{t_{2}}-\hat{y}_{t}|\wedge|\hat{y}_{t}-\hat{y}_{t_{1}}||
	\end{equation*}
	\begin{equation*}
	\leq\sup_{\substack{t_{1}\leq t\leq t_{2}\\t_{2}-t_{1}\leq\delta}}||\hat{x}_{t_{2}}-\hat{x}_{t}|\wedge|\hat{x}_{t}-\hat{x}_{t_{1}}|-|\hat{y}_{t_{2}}-\hat{y}_{t}|\wedge|\hat{y}_{t}-\hat{y}_{t_{1}}||.
	\end{equation*}
	Now, observe that for each triplet $t_{1}\leq t\leq t_{2}$ there are two possible cases: in the first one the minimum is achieved in the same interval, \textit{e.g.}
	\begin{equation*}\label{case-1}
	|\hat{x}_{t_{2}}-\hat{x}_{t}|\wedge|\hat{x}_{t}-\hat{x}_{t_{1}}|=|\hat{x}_{t_{2}}-\hat{x}_{t}|,\quad\textnormal{and}\quad 	|\hat{y}_{t_{2}}-\hat{y}_{t}|\wedge|\hat{y}_{t}-\hat{y}_{t_{1}}|=|\hat{y}_{t_{2}}-\hat{y}_{t}|,
	\end{equation*}
while in the second case the minimum is not achieved in the same interval, \textit{e.g.}
\begin{equation}\label{case-2}
|\hat{x}_{t_{2}}-\hat{x}_{t}|\wedge|\hat{x}_{t}-\hat{x}_{t_{1}}|=|\hat{x}_{t}-\hat{x}_{t_{1}}|,\quad\textnormal{and}\quad 	|\hat{y}_{t_{2}}-\hat{y}_{t}|\wedge|\hat{y}_{t}-\hat{y}_{t_{1}}|=|\hat{y}_{t_{2}}-\hat{y}_{t}|.
\end{equation}
For the first case, following (\ref{case-1}), we have
\begin{equation*}
||\hat{x}_{t_{2}}-\hat{x}_{t}|\wedge|\hat{x}_{t}-\hat{x}_{t_{1}}|-|\hat{y}_{t_{2}}-\hat{y}_{t}|\wedge|\hat{y}_{t}-\hat{y}_{t_{1}}||=||\hat{x}_{t_{2}}-\hat{x}_{t}|-|\hat{y}_{t_{2}}-\hat{y}_{t}||
\end{equation*}
\begin{equation*}
\leq |\hat{x}_{t_{2}}-\hat{x}_{t}-\hat{y}_{t_{2}}+\hat{y}_{t}|\leq |\hat{x}_{t_{2}}-\hat{y}_{t_{2}}|+|\hat{x}_{t}-\hat{y}_{t}|\leq 2\sup_{t}|\hat{x}_{t}-\hat{y}_{t}|=2\|x-y\|_{\infty}.
\end{equation*}
For the second case, following (\ref{case-2}) and considering w.l.o.g.~that $|\hat{x}_{t}-\hat{x}_{t_{1}}|\geq |\hat{y}_{t_{2}}-\hat{y}_{t}|$ we have
\begin{equation*}
||\hat{x}_{t_{2}}-\hat{x}_{t}|\wedge|\hat{x}_{t}-\hat{x}_{t_{1}}|-|\hat{y}_{t_{2}}-\hat{y}_{t}|\wedge|\hat{y}_{t}-\hat{y}_{t_{1}}||=|\hat{x}_{t}-\hat{x}_{t_{1}}|-|\hat{y}_{t_{2}}-\hat{y}_{t}|
\end{equation*}
\begin{equation*}
\leq |\hat{x}_{t_{2}}-\hat{x}_{t}|-|\hat{y}_{t_{2}}-\hat{y}_{t}|\leq |\hat{x}_{t_{2}}-\hat{x}_{t}-\hat{y}_{t_{2}}+\hat{y}_{t}|\leq2\|x-y\|_{\infty}.
\end{equation*}
Therefore, we have that
	\begin{equation*}
	\sup_{\substack{t_{1}\leq t\leq t_{2}\\t_{2}-t_{1}\leq\delta}}||\hat{x}_{t_{2}}-\hat{x}_{t}|\wedge|\hat{x}_{t}-\hat{x}_{t_{1}}|-|\hat{y}_{t_{2}}-\hat{y}_{t}|\wedge|\hat{y}_{t}-\hat{y}_{t_{1}}||\leq 2\|x-y\|_{\infty}.
	\end{equation*}
	Thus, $g_{1}$ is uniformly continuous.
	
	For $g_{2}$ observe that
	\begin{equation*}
	||\hat{x}_{\delta}-\hat{x}_{0}|-|\hat{y}_{\delta}-\hat{y}_{0}||\leq|\hat{x}_{\delta}-\hat{x}_{0}-\hat{y}_{\delta}+\hat{y}_{0}|\leq 2\|x-y\|_{\infty},
	\end{equation*}
	the same arguments apply to $g_{3}$, while for $g_{4}$ we have
	\begin{equation*}
	|\sup_{t}|\hat{x}_{t}|-\sup_{t}|\hat{y}_{t}||\leq \sup_{t}||\hat{x}_{t}|-|\hat{y}_{t}||\leq \sup_{t}|\hat{x}_{t}-\hat{y}_{t}|=2\|x-y\|_{\infty}.
	\end{equation*}
\end{proof}
In the following result, we present the equivalent density result for the class of processes with paths in $D([0,T])$ endowed with Skorokhod topology.
\begin{thm}\label{T-D}
	Let $T>0$. The class of $\Phi$ processes with stepwise paths with steps of equal length is dense in the space of processes with paths in $D([0,T])$ endowed with Skorokhod topology with respect to weak convergence.
\end{thm}
\begin{proof}
	Let $(X_{t})_{t\in[0,1]}$ be any c\'{a}dl\'{a}g process. We focus on the interval $[0,1]$ but the same arguments of the proof apply to any interval $[0,T]$, for $T>0$. Let $\delta_{m}:=\frac{1}{2^{m}}$, $m\in\mathbb{N}$.
	
	For every $n\in\mathbb{N}$, consider a $(2^{n}+1)$-dimensional $\Phi$ random vector 
	\begin{equation*}
	Y^{(n)}:=(Y^{(n)}_{0},Y^{(n)}_{\frac{1}{2^{n}}},Y^{(n)}_{\frac{2}{2^{n}}},...,Y^{(n)}_{1})
	\end{equation*}
	such that
	\begin{equation}\label{X}
	\rho_{2^{n}+1}\left(Y^{(n)},(X_{0},X_{\frac{1}{2^{n}}},...,X_{1}) \right)< \frac{1}{n},
	\end{equation}
	\begin{equation}\label{max-property-Yn-D}
	\max_{m<n}\rho(\sup_{\substack{t_{1}\leq t\leq t_{2}\\|t-s|\leq\delta_{m}}}|\hat{Y}_{t_{2}}^{(n)}-\hat{Y}_{t}^{(n)}|\wedge|\hat{Y}_{t}^{(n)}-\hat{Y}_{t_{1}}^{(n)}|,\sup_{\substack{t_{1}\leq t\leq t_{2}\\|t-s|\leq\delta_{m}}}|\hat{X}_{t_{2}}^{(n)}-\hat{X}_{t}^{(n)}|\wedge|\hat{X}_{t}^{(n)}-\hat{X}_{t_{1}}^{(n)}|)<\frac{1}{n},
	\end{equation}
	\begin{equation*}
	\max_{m<n}\rho(|\hat{Y}_{\delta_{m}}^{(n)}-\hat{Y}_{0}^{(n)}|,|\hat{X}_{\delta_{m}}^{(n)}-\hat{X}_{0}^{(n)}|)<\frac{1}{n},
	\end{equation*}
	\begin{equation*}
	\max_{m<n}\rho(|\hat{Y}_{\frac{2^{n}-1}{2^{n}}}^{(n)}-\hat{Y}_{1-\delta_{m}}^{(n)}|,|\hat{X}_{\frac{2^{n}-1}{2^{n}}}^{(n)}-\hat{X}_{1-\delta_{m}}^{(n)}|)<\frac{1}{n},
	\end{equation*}
	and
	\begin{equation*}
	\rho(\sup_{t}|\hat{Y}_{t}^{(n)}|,\sup_{t}|\hat{X}_{t}^{(n)}|)<\frac{1}{n},
	\end{equation*}
	where
	\begin{equation*}
	\textnormal{$\hat{Y}_{t}^{(n)}:=Y^{(n)}_{\frac{\lfloor t2^{n}\rfloor}{2^{n}}}\quad$ and $\quad\hat{X}_{t}^{(n)}:=X_{\frac{\lfloor t2^{n}\rfloor}{2^{n}}}$},
	\end{equation*}
	for $t\in[0,1]$. The existence of $Y^{(n)}$	is ensured by the density of the family of distributions $\Phi$ and by Lemma \ref{lem-2}. Notice that $(\hat{Y}_{t}^{(n)})_{t\in[0,1]}$ is a c\'{a}dl\'{a}g process due to the property of the floor function, for every $n\in\mathbb{N}$. Indeed, it is possible to see that for any fixed $n\in\mathbb{N}$ and every $s\in[0,\frac{1}{2^{n}})$ we have $\lfloor t2^{n}\rfloor=\lfloor (t+s)2^{n}\rfloor$. Therefore, we have that $\hat{Y}_{t}^{(n)}=\hat{Y}_{t+s}^{(n)}$. Note that this is not true when we look at the left limit.
	
	Now, let $d$ be Skorokhod metric. By Lemma 3 page 127 in \cite{Bill} we have that
	\begin{equation*}
	d(\hat{X}^{(n)},X)\leq 2^{-n}\vee w'_{X}(2^{-n})\stackrel{everywhere}{\to}0,\quad n\to\infty
	\end{equation*}
	where $w'_{X}(2^{-n})$ is the (modified) modulus of continuity for $D[0,1]$. Let $d^{\circ}$ be the metric that makes $D$ complete and separable. Since $d^{\circ}$ and $d$ are equivalent (see Theorem 12.1 in \cite{Bill}), then $d^{\circ}(\hat{X}^{(n)},X)\to0$ everywhere as $n\to\infty$. This implies by Corollary page 28 in \cite{Bill} that 
	\begin{equation}\label{Xn-to-X}
\hat{X}^{(n)}\stackrel{d}{\to}X,\quad\textnormal{as $n\to\infty$}.
	\end{equation}
	Since $D[0,1]$ is complete and separable under $d^{\circ}$, then by the Prokhorov theorem a family of probability measures on $(D,\mathcal{D})$ is relatively compact if and only if it is tight. Let $T_{X}$ denote the set of $t$ in $[0,1]$ for which the natural projection $\pi_{t}$ is continuous except at point forming a set of $P$-measure 0 where $P$ is the distribution of $X$. Then, (\ref{Xn-to-X}) implies that for every $t_{1},....,t_{k}\in T_{X}$, where $k\in\mathbb{N}$, we have the convergence 
	\begin{equation*}
	\rho_{k}\left((\hat{X}^{(n)}_{t_{1}},...,\hat{X}^{(n)}_{t_{k}}),(X_{t_{1}},...,X_{t_{k}})\right)\to0,\quad\textnormal{as $n\to\infty$}.
	\end{equation*}
	Moreover,by (\ref{X}) and by the properties of the Prokhorov metric, for every $t_{1},...,t_{k}\in[0,1]$, we have that 
	\begin{equation*}
	\rho_{k}\left((\hat{Y}^{(n)}_{t_{1}},...,\hat{Y}^{(n)}_{t_{k}}),(\hat{X}^{(n)}_{t_{1}},...,\hat{X}^{(n)}_{t_{k}}) \right)< \frac{1}{n},
	\end{equation*}
	for $n$ large enough, namely for $2^{n}+1\geq k$, and so it converges to zero as $n\to\infty$. Therefore, by the triangular inequality we obtain that for every $t_{1},....,t_{k}\in T_{X}$
	\begin{equation*}
	\rho_{k}\left((\hat{Y}^{(n)}_{t_{1}},...,\hat{Y}^{(n)}_{t_{k}}),(X_{t_{1}},...,X_{t_{k}}) \right)\to0\quad \textnormal{as $n\rightarrow\infty$}.
	\end{equation*}
	and so $(\hat{Y}_{t}^{(n)})_{t\in[0,1]}\stackrel{fdd}{\rightarrow}(X_{t})_{t\in[0,1]}$, as $n\rightarrow\infty$, for points in $T_{X}$.
	
	Let us now prove tightness. Following the arguments above (\ref{Xn-to-X}) implies the tightness of the family of probability distributions of $X_{n}$, $n\in\mathbb{N}$. Thus, by Theorem 13.2 and by the equivalent conditions (13.8) in \cite{Bill}, (\ref{Xn-to-X}) implies that: for every $\gamma>0$ there exists an $a>0$ and a $n_{0}$ such that
	\begin{equation}\label{condition-1}
	\mathbb{P}(\sup_{t}|\hat{X}_{t}^{(n)}|\geq a)\leq\gamma,\quad n\geq n_{0}
	\end{equation}
	and that for every $\eta>0$ and $\varepsilon>0$ there exists a $m\in\mathbb{N}$ and a $\tilde{n}_{0}$ such that
	\begin{equation*}
	\begin{cases}
	\mathbb{P}(\sup_{\substack{t_{1}\leq t\leq t_{2}\\|t-s|\leq\delta_{m}}}|\hat{X}_{t_{2}}^{(n)}-\hat{X}_{t}^{(n)}|\wedge|\hat{X}_{t}^{(n)}-\hat{X}_{t_{1}}^{(n)}|\geq \varepsilon)\leq\eta,\\
	\mathbb{P}(|\hat{X}_{\delta_{m}}^{(n)}-\hat{X}_{0}^{(n)}|\geq \varepsilon)\leq\eta,\\
	\mathbb{P}(|\hat{X}_{\frac{2^{n}-1}{2^{n}}}^{(n)}-\hat{X}_{1-\delta_{m}}^{(n)}|\geq \varepsilon)\leq\eta,
	\end{cases}
	\end{equation*}
	for $n\geq \tilde{n}_{0}$.
	
	Thus, by Theorem 13.2 and by the equivalent conditions (13.8) in \cite{Bill} in order to show tightness we need to show that these properties are satisfied by $((\hat{Y}_{t}^{(n)})_{t\in[0,1]})_{n\in\mathbb{N}}$. Following the arguments in the proof of Theorem \ref{T-C} we have the following. Fix $\eta$ and $\varepsilon$. Let $m$ and $\tilde{n}_{0}$ be such that 
	\begin{equation*}
	\begin{cases}
	\mathbb{P}(\sup_{\substack{t_{1}\leq t\leq t_{2}\\|t-s|\leq\delta_{m}}}|\hat{X}_{t_{2}}^{(n)}-\hat{X}_{t}^{(n)}|\wedge|\hat{X}_{t}^{(n)}-\hat{X}_{t_{1}}^{(n)}|< \frac{\varepsilon}{2})\geq1-\frac{\eta}{2},\\
	\mathbb{P}(|\hat{X}_{\delta_{m}}^{(n)}-\hat{X}_{0}^{(n)}|< \frac{\varepsilon}{2})\geq1-\frac{\eta}{2},\\
	\mathbb{P}(|\hat{X}_{\frac{2^{n}-1}{2^{n}}}^{(n)}-\hat{X}_{1-\delta_{m}}^{(n)}|< \frac{\varepsilon}{2})\geq1-\frac{\eta}{2},
		\end{cases}
	\end{equation*}	
	for $n\geq \tilde{n}_{0}$. Then, for every $n>\max(m,\tilde{n}_{0},\frac{2}{\varepsilon},\frac{2}{\eta})$
	\begin{equation*}
	\begin{cases}
	\mathbb{P}(\sup_{\substack{t_{1}\leq t\leq t_{2}\\|t-s|\leq\delta_{m}}}|\hat{Y}_{t_{2}}^{(n)}-\hat{Y}_{t}^{(n)}|\wedge|\hat{Y}_{t}^{(n)}-\hat{Y}_{t_{1}}^{(n)}|< \varepsilon)\geq1-\eta,\\
	\mathbb{P}(|\hat{Y}_{\delta_{m}}^{(n)}-\hat{Y}_{0}^{(n)}|< \varepsilon)\geq1-\eta,\\
	\mathbb{P}(|\hat{Y}_{\frac{2^{n}-1}{2^{n}}}^{(n)}-\hat{Y}_{1-\delta_{m}}^{(n)}|< \varepsilon)\geq1-\eta.
	\end{cases}
	\end{equation*}
	Concerning condition (\ref{condition-1}) we have the following. Fix $\gamma$. Let $a>0$ and $n_{0}$ be such that
	\begin{equation*}
	\mathbb{P}(\sup_{t}|\hat{X}_{t}^{(n)}|\geq a)\leq\frac{\gamma}{2},\quad n\geq n_{0}.
	\end{equation*}
	Then, for every $n>\max(n_{0},\frac{2}{\gamma})$ we have
		\begin{equation*}
		\mathbb{P}(\sup_{t}|\hat{Y}_{t}^{(n)}|< a+1)\geq\mathbb{P}(\sup_{t}|\hat{X}_{t}^{(n)}|< a+1-\frac{1}{n})-\frac{1}{n}\geq1-\frac{\gamma}{2}-\frac{1}{n}\geq1-\gamma.
		\end{equation*}
		
		Then, by Theorem 13.1 in \cite{Bill} we obtain the result.
\end{proof}

\begin{lem}\label{lem-3}
	Let $d\in\mathbb{N}$. Let $Z^{(n)}$, $n\in\mathbb{N}$, and $Z$ be random vectors on $\mathbb{R}^{d2^d+1}$ such that $Z^{(n)}\stackrel{d}{\to}Z$ as $n\to\infty$. Denote the elements of $Z^{(n)}$ as follows $Z^{(n)}=(Z^{(n)}_{0},Z^{(n)}_{\frac{1}{2^{d}}},Z^{(n)}_{\frac{2}{2^{d}}},...,Z^{(n)}_{1},$ $Z^{(n)}_{1+\frac{1}{2^{d}}},...,Z^{(n)}_{d})$. Let $T\leq d$ and let for $t\in[0,T]$, $\hat{Z}_{t}^{(n)}:=Z^{(n)}_{\frac{\lfloor t2^{d}\rfloor}{2^{d}}}$ and $\hat{Z}_{t}:=Z_{\frac{\lfloor t2^{d}\rfloor}{2^{d}}}$. Then, for every fixed $0<\delta<T$, we have that
	\begin{equation*}
	\sup_{\substack{t_{1}\leq t\leq t_{2} \\t_{2}-t_{1}\leq\delta\\t_{1},t_{2}\in[0,T]}}|\hat{Z}_{t_{2}}^{(n)}-\hat{Z}_{t}^{(n)}|\wedge|\hat{Z}_{t}^{(n)}-\hat{Z}_{t_{1}}^{(n)}|\stackrel{d}{\to}	\sup_{\substack{t_{1}\leq t\leq t_{2}\\t_{2}-t_{1}\leq\delta\\t_{1},t_{2}\in[0,T]}}|\hat{Z}_{t_{2}}-\hat{Z}_{t}|\wedge|\hat{Z}_{t}-\hat{Z}_{t_{1}}|,
	\end{equation*}
	\begin{equation*}
	\sup_{t\in[0,T]}|\hat{Z}_{t}^{(n)}|\stackrel{d}{\to}	\sup_{t\in[0,T]}|\hat{Z}_{t}|.
	\end{equation*}
	Further, for every fixed $0<\delta<T$ we have that
	\begin{equation*}
	\sup_{\substack{t_{1}\leq t\leq t_{2}\\t_{2}-t_{1}\leq\delta\\t_{1},t_{2}\in[0,T]}}|\hat{Z}_{t_{2}}-\hat{Z}_{t}|\wedge|\hat{Z}_{t}-\hat{Z}_{t_{1}}|\stackrel{everywhere}{=}\sup_{\substack{t_{1}\leq t\leq t_{2}\\t_{2}-t_{1}\leq\delta\wedge\frac{\lfloor T2^{d}\rfloor}{2^{d}}\\t_{1},t_{2}\in[0,\frac{\lfloor T2^{d}\rfloor}{2^{d}}]}}|\hat{Z}_{t_{2}}-\hat{Z}_{t}|\wedge|\hat{Z}_{t}-\hat{Z}_{t_{1}}|,
	\end{equation*}
	\begin{equation*}
	\sup_{t\in[0,T]}|\hat{Z}_{t}|\stackrel{everywhere}{=}\sup_{t\in[0,\frac{\lfloor T2^{d}\rfloor}{2^{d}}]}|\hat{Z}_{t}|
	\end{equation*}
	and that
	\begin{equation*}
	|\hat{Z}_{\delta}-\hat{Z}_{0}|\stackrel{everywhere}{=}|Z_{p}-Z_{q}|,\quad\quad |\hat{Z}_{T-}-\hat{Z}_{T-\delta}|\stackrel{everywhere}{=}|Z_{i}-Z_{j}|,
	\end{equation*}
	for some $p,q,i,j\in\{0,\frac{1}{2^{d}},\frac{2}{2^{d}},...,d\}$.
\end{lem}
\begin{proof}
	The first statement follows from the same arguments as the ones used in the proof of Lemmas \ref{lem-1} and \ref{lem-2}. The second statement follows immediately from the stepwise structure of $\hat{Z}$.
\end{proof}
In the next result we are going to prove the density result in the space of processes with paths in $D[0,\infty)$ endowed with the Skorokhod topology.
\begin{thm}\label{T-D-inf}
	The class of $\Phi$ processes with stepwise paths is dense in the space of processes with paths in $D[0,\infty)$ endowed with the Skorokhod topology with respect to weak convergence.
\end{thm}
\begin{proof}
	Let $(X_{t})_{t\in[0,1]}$ be any c\'{a}dl\'{a}g process. Let $\delta_{m}:=\frac{1}{2^{m}}$, $m\in\mathbb{N}$. For every $n\in\mathbb{N}$, consider a $(n2^{n}+1)$-dimensional $\Phi$ random vector 
	\begin{equation*}
	Y^{(n)}:=(Y^{(n)}_{0},Y^{(n)}_{\frac{1}{2^{n}}},Y^{(n)}_{\frac{2}{2^{n}}},...,Y^{(n)}_{1},Y^{(n)}_{1+\frac{1}{2^{n}}},...,Y^{(n)}_{n})
	\end{equation*}
	such that
	\begin{equation*}
	\rho_{n2^{n}+1}\left(Y^{(n)},(X_{0},X_{\frac{1}{2^{n}}},...,X_{n}) \right)< \frac{1}{n},
	\end{equation*}
	\begin{equation*}
	\sup_{T\in[0,n]}\max_{m<n}\rho(\sup_{\substack{t_{1}\leq t\leq t_{2}\\t_{2}-t_{1}\leq\delta_{m}\\t_{1},t_{2}\in[0,T]}}|\hat{Y}_{t_{2}}^{(n)}-\hat{Y}_{t}^{(n)}|\wedge|\hat{Y}_{t}^{(n)}-\hat{Y}_{t_{1}}^{(n)}|,\sup_{\substack{t_{1}\leq t\leq t_{2}\\t_{2}-t_{1}\leq\delta_{m}\\t_{1},t_{2}\in[0,T]}}|\hat{X}_{t_{2}}^{(n)}-\hat{X}_{t}^{(n)}|\wedge|\hat{X}_{t}^{(n)}-\hat{X}_{t_{1}}^{(n)}|)<\frac{1}{n},
	\end{equation*}
	\begin{equation*}
	\sup_{T\in[0,n]}\rho(\sup_{t\in[0,T]}|\hat{Y}_{t}^{(n)}|,\sup_{t\in[0,T]}|\hat{X}_{t}^{(n)}|)<\frac{1}{n},
	\end{equation*}
	\begin{equation*}
	\max_{i,j\in\{0,\frac{1}{2^{n}},\frac{2}{2^{n}},...,n\}}\rho(|Y^{(n)}_{i}-Y^{(n)}_{j}|,|X_{i}-X_{j}|)<\frac{1}{n},
	\end{equation*}
	where
	\begin{equation*}
	\textnormal{$\hat{Y}_{t}^{(n)}:=\begin{cases}
		Y^{(n)}_{\frac{\lfloor t2^{n}\rfloor}{2^{n}}}&\textnormal{for $t\in[0,n]$},\\Y^{(n)}_{n}&\textnormal{for $t\in(n,\infty)$},
		\end{cases}\quad$ and $\quad\hat{X}_{t}^{(n)}:=\begin{cases}
		X_{\frac{\lfloor t2^{n}\rfloor}{2^{n}}}&\textnormal{for $t\in[0,n]$},\\X_{n}&\textnormal{for $t\in(n,\infty)$}.
		\end{cases}$}
	\end{equation*}
Observe that $Y^{(n)}$ exists because of the density of the family of distributions $\Phi$ and because of Lemma \ref{lem-3}. In particular, by Lemma \ref{lem-3} we have that
\begin{equation*}
\sup_{T\in[0,n]}\max_{m<n}\rho(\sup_{\substack{t_{1}\leq t\leq t_{2}\\t_{2}-t_{1}\leq\delta_{m}\\t_{1},t_{2}\in[0,T]}}|\hat{Y}_{t_{2}}^{(n)}-\hat{Y}_{t}^{(n)}|\wedge|\hat{Y}_{t}^{(n)}-\hat{Y}_{t_{1}}^{(n)}|,\sup_{\substack{t_{1}\leq t\leq t_{2}\\t_{2}-t_{1}\leq\delta_{m}\\t_{1},t_{2}\in[0,T]}}|\hat{X}_{t_{2}}^{(n)}-\hat{X}_{t}^{(n)}|\wedge|\hat{X}_{t}^{(n)}-\hat{X}_{t_{1}}^{(n)}|)
\end{equation*}
\begin{equation*}
=\max_{T\in\{0,\frac{1}{2^{n}},\frac{2}{2^{n}},...,n\}}\max_{m<n}\rho(\sup_{\substack{t_{1}\leq t\leq t_{2}\\t_{2}-t_{1}\leq\delta_{m}\\t_{1},t_{2}\in[0,T]}}|\hat{Y}_{t_{2}}^{(n)}-\hat{Y}_{t}^{(n)}|\wedge|\hat{Y}_{t}^{(n)}-\hat{Y}_{t_{1}}^{(n)}|,\sup_{\substack{t_{1}\leq t\leq t_{2}\\t_{2}-t_{1}\leq\delta_{m}\\t_{1},t_{2}\in[0,T]}}|\hat{X}_{t_{2}}^{(n)}-\hat{X}_{t}^{(n)}|\wedge|\hat{X}_{t}^{(n)}-\hat{X}_{t_{1}}^{(n)}|),
\end{equation*}
and
\begin{equation*}
\sup_{T\in[0,n]}\rho(\sup_{t\in[0,T]}|\hat{Y}_{t}^{(n)}|,\sup_{t\in[0,T]}|\hat{X}_{t}^{(n)}|)=\max_{T\in\{0,\frac{1}{2^{n}},\frac{2}{2^{n}},...,n\}}\rho(\sup_{t\in[0,T]}|\hat{Y}_{t}^{(n)}|,\sup_{t\in[0,T]}|\hat{X}_{t}^{(n)}|).
\end{equation*}
Considering the proof of Theorem \ref{T-D}, it is immediate to see that $(\hat{Y}_{t}^{(n)})_{t\in[0,\infty)}$ is a c\'{a}dl\'{a}g process.

Now, let $x$ be an element of $D[0,\infty)$ and let $r_{t}x$ be the restriction of $x$ to $[0,t]$. As shown in page 174 in \cite{Bill} $r_{t}$ is $D[0,\infty)/D[0,t]$ measurable. Notice that for any $T\in\mathbb{N}$ by Lemma 3 page 127 in \cite{Bill} we have that, for every $n\geq T$, $d_{T}(r_{T}\hat{X}^{(n)},r_{T}X)\leq 2^{-n}\vee w'_{r_{T}X}(2^{-n})$ and so
\begin{equation*}
d_{T}^{\circ}(r_{T}\hat{X}^{(n)},r_{T}X)\stackrel{everywhere}{\to}0,\quad n\to\infty.
\end{equation*}
Then, by Lemma 1 in page 167 in \cite{Bill} we have that, for every $\omega\in\Omega$ and $s\in[0,T]$ such that $(X(\omega)_{t\in[0,T]})$ is continuous at $s$,
\begin{equation*}
d_{s}^{\circ}(r_{s}\hat{X}^{(n)}(\omega),r_{s}X(\omega))\to0,\quad n\to\infty.
\end{equation*}
This implies that for every $s$ in $T_{X}$, namely the set of continuity point of $X$, and such that $s\leq T$, we have that
\begin{equation*}
d_{s}^{\circ}(r_{s}\hat{X}^{(n)},r_{s}X)\stackrel{a.s.}{\to}0,\quad n\to\infty.
\end{equation*}
Since $T$ was arbitrary, we have shown that for every continuity point $z$ of $X$ we have that
\begin{equation*}
d_{z}^{\circ}(r_{z}\hat{X}^{(n)},r_{z}X)\stackrel{a.s.}{\to}0,\quad n\to\infty,
\end{equation*}
which by Corollary of page 28 in \cite{Bill} implies that $r_{z}\hat{X}^{(n)}\stackrel{d}{\to}r_{z}X$, as $n\to\infty$.

Let $z\in T_{X}$. Following the same arguments as the ones used in the proof of Theorem \ref{T-D} for the fdd convergence, we obtain  that
\begin{equation*}
r_{z}\hat{Y}^{(n)}\stackrel{fdd}{\to}r_{z}X,\quad n\to\infty.
\end{equation*} 
Moreover, following the same arguments as the ones used in the proof of Theorem \ref{T-D} for tightness, with the additional constraint that both $n_{0}$ and $\tilde{n}_{0}$ must also be greater or equal than $z$, we obtain the tightness of the family of probability distributions of $r_{z}Y_{n}$, $n\in\mathbb{N}$. Then, by Theorem 13.1 in \cite{Bill} we obtain that $r_{z}\hat{Y}^{(n)}\stackrel{d}{\to}r_{z}X$, as $n\to\infty$.

Since the arguments hold for every $z\in T_{X}$, by Theorem 16.7 in \cite{Bill} we obtain the result.
\end{proof}
Thanks to this result we are able to provide further classes of dense stochastic processes in the space of processes with paths in $D[0,T]$ endowed with the Skorokhod topology. The following results differ from Theorem \ref{T-D} on how we consider the last step of the approximating step-wise processes.
\begin{co}
The set of approximating stochastic processes in the proof of Theorem \ref{T-D-inf} truncated at time $T$ is dense in the space of processes with paths in $D[0,T]$ endowed with the Skorokhod topology with respect to weak convergence.
\end{co}
\begin{proof}
	This is an application of Theorem 16.7 in \cite{Bill} in combination with Theorem \ref{T-D-inf}.
\end{proof}
Let $m\in\mathbb{N}$ and let
\begin{equation*}
g_m(t)=\begin{cases}
1&\textnormal{ if $t\leq m-1$},\\
m-t&\textnormal{ if $m-1\leq t\leq m$},\\
0&\textnormal{ if $t\geq m$}.
\end{cases}
\end{equation*}
Further, for every path $x\in D[0,\infty)$ let $\psi_mx(t):=g_m(t)x(t)$, for $t\in[0,m]$. Observe that $\psi_mx\in D[0,m]$. Let $\tilde{D}[0,m]$ the set of paths $\psi_mx$ for every $x\in D[0,\infty)$. Notice that $\tilde{D}[0,m]\subset D[0,m]$.
\begin{co}
	Let $m\in\mathbb{N}$. Apply $\psi_m$ to the approximating stochastic processes in the proof of Theorem \ref{T-D-inf}. Then this class of stochastic processes is dense in the space of processes with paths in $\tilde{D}[0,m]$ endowed with the Skorokhod topology with respect to weak convergence.
\end{co}
\begin{proof}
	This is an application of Lemma 3 page 173 in \cite{Bill} in combination with Theorem \ref{T-D-inf}.
\end{proof}
All the results presented in this section show the following remark: once we have a dense class of probability distributions then it is possible to construct various dense classes of stochastic processes in different topological frameworks.
\section{The QID case}\label{Sec-QID}
In this section we investigate the implication of the above results in the QID case. First, we introduce the following conjecture.
\begin{conj}\label{T1}
	Let $d\in\mathbb{N}$. The class of QID distributions on $\mathbb{R}^{d}$ with finite
	quasi-L\'{e}vy measure and zero Gaussian variance is dense in the space of probability distributions on $\mathbb{R}^{d}$ with respect to weak convergence.
\end{conj}
Now, we present the density results for QID processes in the the various settings discussed in the previous section.

\textit{For all the next results we assume that the Conjecture \ref{T1} is true}.
\begin{co}\label{0-co}
	The class of QID time series s.t.~their fdd~have finite quasi-L\'{e}vy measure and zero Gaussian variance is dense in the space of all time series with respect to the fdd convergence.
\end{co}
\begin{proof}
	It follows from Conjecture \ref{T1} and Proposition \ref{0}.
\end{proof}

\begin{co}\label{T-C-co}
	Let $T>0$. The class of QID processes with linear piecewise paths and s.t.~their fdd have finite quasi-L\'{e}vy measure and zero Gaussian variance is dense in the space of processes with paths in $C([0,T])$ endowed with the uniform topology with respect to weak convergence.
\end{co}
\begin{proof}
	Observe that, for every $n\in\mathbb{N}$, the stochastic process $(\hat{Y}_{t}^{(n)})_{t\in[0,1]}$ is a QID process since all its finite dimensional distributions are QID (see Definition 7.1 in \cite{Passeggeri}). Indeed, this is because given a QID random vector $Z$ in $\mathbb{R}^{p}$ for every matrix $U\in\mathbb{R}^{q\times p}$ we have that $UZ$ is a QID random vector in $\mathbb{R}^{q}$, for every $p,q\in\mathbb{N}$ (see Proposition 11.10 in \cite{Sato} for the ID case -- the QID case is identical). Then, the result follows from Conjecture \ref{T1} and Theorem \ref{T-C}.
\end{proof}
\begin{co}\label{T-D-co}
	Let $T>0$. The class of QID processes with stepwise paths, with steps of equal length, and s.t.~their fdd have finite quasi-L\'{e}vy measure and zero Gaussian variance is dense in the space of processes with paths in $D([0,T])$ endowed with the Skorokhod topology with respect to weak convergence.
\end{co}
\begin{proof}
	It follows from Conjecture \ref{T1} and Theorem \ref{T-D}.
\end{proof}
\begin{co}\label{T-D-inf-co}
	The class of QID processes with stepwise paths and s.t.~their fdd have finite quasi-L\'{e}vy measure and zero Gaussian variance is dense in the space of processes with paths in $D[0,\infty)$ endowed with the Skorokhod topology with respect to weak convergence.
\end{co}
\begin{proof}
	It follows from Conjecture \ref{T1} and Theorem \ref{T-D-inf}.
\end{proof}
\section*{Acknowledgement}
The author would like to thank Alexander Lindner and David Berger for useful discussions. The research developed in this paper has been supported by the EPSRC (award ref.~1643696) at Imperial College London and by the Fondation Sciences Math\'{e}matiques de Paris (FSMP) fellowship, held at LPSM (Sorbonne University).

\end{document}